\documentclass[11pt]{amsart}

\usepackage{amsmath, amssymb, verbatim}
\usepackage{amsthm}

\usepackage{setspace}
\newtheorem{theorem}{Theorem}[section]
\newtheorem{lemma}[theorem]{Lemma}
\newtheorem{cor}[theorem]{Corollary}

\DeclareMathOperator{\E}{\mathbb{E}}

\DeclareMathOperator{\trace}{Tr}
\DeclareMathOperator{\range}{range}
\DeclareMathOperator{\spanof}{span}
\DeclareMathOperator{\vol}{vol}

\DeclareMathOperator{\disc}{disc}
\DeclareMathOperator{\vecdisc}{vecdisc}

\newcommand{\vect}[1]{\mathbf{#1}}
\newcommand{\psd}{\succeq} %

\newcommand{\junk}[1]{}

\begin{document}
\title[Koml\'{o}s Vector Discrepancy]{The Koml\'{o}s Conjecture Holds
  for Vector Colorings} 

\maketitle

\begin{abstract}
  The Koml\'{o}s conjecture in discrepancy theory states that for some  constant $K$ and for any $m \times n$ matrix $\vect{A}$ whose  columns lie in the unit ball there exists a vector $\vect{x} \in  \{-1, +1\}^n$ such that $\|\vect{Ax}\|_\infty \leq K$. This  conjecture also implies the Beck-Fiala conjecture on the discrepancy  of bounded degree hypergraphs. Here we prove a natural relaxation of  the Koml\'{o}s conjecture: if the columns of $\vect{A}$ are assigned  unit vectors in $\mathbb{R}^n$ rather than $\pm 1$ then the  Koml\'{o}s conjecture holds with $K=1$. Our result rules out the  possibility of a counterexample to the conjecture based on the  natural semidefinite relaxation of discrepancy. It also opens the  way to proving tighter efficient (polynomial-time computable) upper  bounds for the conjecture using semidefinite programming techniques.
\end{abstract}

\section{Introduction}

Let $\mathcal{H} = \{H_1, \ldots, H_m\}$ be a hypergraph with vertex
set $V = [n]$. In this work we study the \emph{combinatorial
  discrepancy} of hypergraphs \junk{--- a combinatorial quantity with
applications in geometry, computer science, and numerical integration,
among others ---} and related quantities. The discrepancy of
$\mathcal{H}$ is defined as
\begin{equation}
  \disc(\mathcal{H}) = \min_{\chi:[n] \rightarrow \{-1, +1\}} \max_{i =
    1}^m \left|\sum_{j \in H_i}{\chi(j)}\right|.
\end{equation}
Intuitively, discrepancy is the optimization problem of coloring the
vertices of a hypergraph, so that the most imbalanced edge is as
balanced as possible. Thus discrepancy is intimately connected to
problems in Ramsey theory that study conditions under which every
coloring leaves some edge monochromatic.  Discrepancy 
has applications in geometry, computer science, and numerical
integration, among others --- the books by
Matou\v{s}ek~\cite{matousek2010geometric}, Chazelle~\cite{Chazelle},
and the chapter by Beck and S\'{o}s~\cite{beck1996discrepancy} provide
references for a wide array of applications.

We will be particularly interested in the discrepancy of hypergraphs
with maximum degree bounded above by a parameter $t$, i.e.~hypergraphs
$\mathcal{H}$ all of whose vertices appear in at most $t$ edges. It is
a classical result of Beck and Fiala~\cite{beckfiala}  that for any
$\mathcal{H}$ of maximum degree at most $t$, $\disc(\mathcal{H}) \leq
2t-1$. Furthermore, they conjectured that $\disc(\mathcal{H}) \leq
C\sqrt{t}$ for an absolute constant $C$. Proving Beck and Fiala's
conjecture remains an elusive open problem in discrepancy theory.

As usual, we define the incidence matrix of $\mathcal{H}$ as an $m \times
n$ $0$-$1$ matrix $\vect{A}$ such that $A_{ij} = 1$ if and only if $j
\in H_i$. In matrix notation discrepancy can be defined as
$\disc(\mathcal{H}) = \min_{\vect{x} \in \{-1,
  1\}^n}{\|\vect{Ax}\|_\infty}$. This algebraic formulation allows us
to extend the definition of discrepancy to arbitrary matrices:
$\disc(\vect{A}) = \min_{\vect{x} \in \{-1,
  1\}^n}{\|\vect{Ax}\|_\infty}$. Interpreted in this way, discrepancy
is a vector balancing problem:  our goal is to assign signs to a given
set of $n$ vectors (the columns of $\vect{A}$), so that the signed sum
has small norm (infinity norm in our case). A natural restriction on
$\vect{A}$, analogous to the maximum degree restriction for
hypergraphs, is to bound the maximum of some norm of the columns of
$\vect{A}$. Such vector balancing problems were first considered in a
general form by B\'{a}r\'{a}ny and Grinberg~\cite{baranygrinberg},
although a similar problem was posed as early as 1963 by
Dworetzky. The proof of Beck and Fiala shows that for any $\vect{A}$
whose columns have $\ell_1$ norm at most $1$, $\disc(\vect{A}) \leq
2$. Koml\'{o}s conjectured\footnote{The earliest reference we can find
  is the 1987 book `Ten Lectures on the Probabilistic Method' by
  Spencer~\cite{tenlectures}} that for $\vect{A}$ whose columns have
$\ell_2$ norm at most $1$, $\disc(\vect{A}) \leq K$ for some absolute
constant $K$. The Koml\'{o}s conjecture implies the Beck-Fiala
conjecture and also remains open. The best partial progress towards proving
the Koml\'{o}s conjecture is a result by Banaszczyk~\cite{bana}, who
showed the bound $\disc(\vect{A}) \leq K\sqrt{\log n}$ for an absolute
constant $K$. This is the best known bound for the Beck-Fiala
conjecture as well. 

In this paper we are concerned with a natural convex relaxation of
discrepancy: \emph{vector discrepancy}. Vector discrepancy is defined
analogously to discrepancy, but we ``color'' $[n]$ with unit
$n$-dimensional vectors rather than $\pm 1$:
\begin{equation}
  \vecdisc(\vect{A}) = \min_{\vect{u_1}, \ldots, \vect{u_n} \in S^{n-1}} {\max_{i =
    1}^m{\left\|\sum_{j = 1}^n{A_{ij}\vect{u_j}}\right\|_2}},
\end{equation}
where $S^{n-1}$ is the unit sphere in $\mathbb{R}^n$. Vector
discrepancy is a relaxation of discrepancy, i.e. $\vecdisc(\vect{A})
\leq \disc(\vect{A})$ for all matrices $\vect{A}$: a coloring
$\vect{x}$ achieving $\disc(\vect{A})$ induces a vector coloring
$\{\vect{u_i} = x_i \vect{e_1}\}_{i = 1}^n$ ($\vect{e_i}$ being the
$i$-th standard basis vector) achieving vector discrepancy with the
same value. Vector discrepancy was used by Lov\'{a}sz to give an
alternative proof of Roth's lower bound on the discrepancy of
arithmetic progressions~\cite{lovasz2000integer}. A natural question
is whether a lower bound on vector discrepancy could disprove the
Koml\'{o}s conjecture. Our main result is a negative answer to this
question.
\begin{theorem}\label{thm:main}
  For any $m\times n$ real matrix $\vect{A}$ whose columns have
  $\ell_2$ norm at most $1$, $\vecdisc(\vect{A}) \leq 1$. 
\end{theorem}
This theorem  is an analog of the Koml\'{o}s conjecture for vector
discrepancy. 

Except as a means to lower bound discrepancy, vector discrepancy has
also recently proved itself useful in establishing \emph{efficient} upper
bounds on discrepancy. In a recent breakthrough, Bansal~\cite{nikhil}
showed the following theorem. 
\begin{theorem}[\cite{nikhil}]\label{thm:nikhil}
  Let $\vect{A}$ be a real $m \times n$ matrix and assume that for any
  submatrix $\vect{B}$ of $\vect{A}$ we have $\vecdisc(\vect{B}) \leq
  D$. Then $\disc(\vect{A}) \leq D \cdot K \log m$, and, furthermore,
  there exists a polynomial time randomized algorithm which on input
  $\vect{A}$ outputs $\vect{x} \in \{-1, 1\}^n$ such that, with high
  probability, $\|\vect{Ax}\|_\infty \leq D \cdot K \log m$ for an
  absolute constant $K$.\qed
\end{theorem}
In light of Bansal's result, Theorem~\ref{thm:main} implies that for
any $\vect{A}$ whose columns lie in the unit ball
$\disc(\vect{A}) \leq K\log m$ and that a coloring $\vect{x}$
achieving this bound can be found in randomized polynomial time. Such an efficient
upper bound for the Koml\'{o}s conjecture was proved by Bansal~\cite{nikhil}, and
later using different methods by Lovett and
Meka~\cite{lovettmeka}. However, Bansal's, and Lovett and Meka's upper
bounds are based on the ``partial coloring'' method and a $\log n$
factor seems inherent to upper bounds for the Koml\'{o}s conjecture
derived using this method. On the other hand,
Matou\v{s}ek~\cite{matousek-det} conjectures that the $\log m$ factor
in Theorem~\ref{thm:nikhil} can be improved to $\sqrt{\log m}$. If
this conjecture holds, we would have an alternative, and
\emph{efficient} proof of Banaszczyk's upper bound. We note that
Banaszczyk's proof does not obviously yield an efficient algorithm,
and no polynomial time algorithm that matches his bound is currently known.

To the best of our knowledge, Theorem~\ref{thm:main} establishes the
first constant upper bound on the vector discrepancy of matrices with
bounded column $\ell_2$ norms and on the vector discrepancy of bounded
degree hypergraphs. A weaker bound of $O(\sqrt{\log m})$ can be derived in a
variety of ways: directly from Banaszczyk's upper bound; from the
existence of constant discrepancy partial colorings for the Koml\'{o}s
conjecture; from Matou\v{s}ek's recent upper bound~\cite{matousek-det}
on vector discrepancy in terms of the determinant lower bound of
Lov\'{a}sz, Spencer, and Vesztergombi~\cite{determinant-lb}.  Our bound is
tight, as $\vecdisc((1)) = 1$, for example.

\textbf{Techniques.} Our proof of Theorem~\ref{thm:main} relies on a
dual characterization of vector discrepancy, first used by
Matou\v{s}ek to show that the determinant lower bound on discrepancy
is almost tight~\cite{matousek-det}. However, our result does not
follow directly from Matou\v{s}ek's techniques, which only imply a
bound of $O(\sqrt{\log m})$. Vector discrepancy is equivalent to a
semidefinite programming problem, and, using a variant of the Farkas
lemma for semidefinite programming, we can can formulate a dual
program which is feasible for a parameter $D$ precisely when
$\vecdisc(\vect{A}) \geq D$. We assume that the dual program is
feasible for $D = 1 + \epsilon$. Geometrically, this feasibility can
be formulated as the existence of two ellipsoids $E$ and $F$ such that
$F \subseteq E$ and the sum of squared axes lengths of $E$ is at most
a $D$ factor larger than the sum of squared axes lengths of
$\vect{A}F$. The containment $F \subseteq E$ implies that the largest
$k$-dimensional section of $E$ has volume lowerbounded by the largest
$k$-dimensional section of $F$, for all $k$. Since the columns of
$\vect{A}$ lie inside the unit ball, Hadamard's bound then implies
that the axes lengths of $E$ multiplicatively majorize the axes
lengths of $\vect{A}F$, and, by Schur convexity, we have a contradiction to
the assumed constraints on the axes lengths of $E$ and $\vect{A}F$.

\section{Preliminaries}

In this section we introduce some basic notation and useful linear algebraic
facts. 

\subsection{Notation}

We use boldface to denote  matrices: $\vect{A}$, $\vect{X}$. We
denote the entry in the $i$-th row and $j$-the column of $\vect{A}$ as
$A_{ij}$. We denote by $\range(\vect{A})$ the vector space spanned by
the columns of $\vect{A}$, and by $\ker(\vect{A})$ the kernel
(nullspace) of $\vect{A}$. We'll assume a generic matrix $\vect{A}$
has dimensions $m$ by $n$. By $\|\cdot\|$ we denote the standard
$\ell_2$ norm. 

For a real symmetric matrix $\vect{X}$, we use $\vect{X} \psd 0$ to
denote that $\vect{X}$ is positive semidefinite.\junk{, i.e.
\begin{equation}
  \vect{X} \psd 0\Leftrightarrow \forall \vect{y}: \vect{y^TAy} \geq 0.
\end{equation}}

For a real $m$ by $n$ matrix $\vect{A}$, we define the \emph{discrepancy} of $\vect{A}$ as 
\begin{equation}
  \disc(\vect{A}) = \min_{x \in \{-1, 1\}^n}{\|\vect{Ax}\|_\infty}.
\end{equation}
We define the \emph{vector discrepancy} of $\vect{A}$ as
\begin{equation}
  \vecdisc(\vect{A}) = \min_{\vect{u_1}, \ldots, \vect{u_n} \in S^{n-1}}{\max_{i =
    1}^m{\left\|\sum_{j = 1}^n{A_{ij}\vect{u_j}}\right\|_2}},
\end{equation}
where $S^{n-1}$ is the $(n-1)$-dimensional unit sphere in
$\mathbb{R}^n$.  As noted earlier, $\vecdisc(\vect{A}) \leq
\disc(\vect{A})$ for all $\vect{A}$.

%We also define the dot product of two $n \times n$ matrices $A \matdot B$
%as $\sum_{i = 1}^n{\sum_{j = 1}^n{a_{ij}b_{ij}}}$. 

%We define the $\ell_1$ norm of an $n\times n$ matrix $A$ as $\|A\|_1 =
%\sum_{i, j}{|a_{ij}|}$. 

\subsection{Dual Characterization of Vector Discrepancy}

For each matrix $\vect{A}$, $\vecdisc(\vect{A})$ is defined as the
minimum value of a convex function over a convex set, i.e.~as the
value of a convex optimization problem. In particular,
$\vecdisc(\vect{A})^2$ can be written as the optimal solution to the
\emph{semidefinite program}
\begin{align}
  &\min D\label{eq:SDP-beg}\\ 
  &\text{~~~subject to}\\
  &\forall 1\leq i\leq m: (\vect{AXA^T})_{ii} \leq D\\
  &\forall 1\leq i \leq n: X_{ii} = 1\\
  &\vect{X} \psd 0.\label{eq:SDP-end}
\end{align}
To see the equivalence, write the vectors $\vect{u_1}, \ldots,
\vect{u_n}$ forming a vector coloring as the columns of the matrix
$\vect{U}$ and set $\vect{X} = \vect{U^TU} \psd 0$. Also, by the Cholesky
decomposition of positive semidefinite matrices, any $\vect{X} \psd 0$
can be written as $\vect{X} = \vect{U^TU}$ where the columns of
$\vect{U}$ are unit vectors and therefore give a vector coloring. 

Using strong duality for convex programming, we can derive the dual
program to (\ref{eq:SDP-beg})--(\ref{eq:SDP-end}) and characterize the
squared vector discrepancy of $\vect{A}$ as the optimal (maximum)
solution to this dual. A derivation of the dual appears in recent work
by Matou\v{s}ek~\cite{matousek-det}. Next we present the resulting
characterization of vector discrepancy. For a detailed proof of
Theorem~\ref{thm:vecdisc-unif}, see~\cite{matousek-det}.

\begin{theorem}[\cite{matousek-det}]\label{thm:vecdisc-unif}
  For any real $m \times n$ matrix $\vect{A}$, 
  \begin{equation}
    \vecdisc(\vect{A}) \geq D
  \end{equation}
 if and only there exists a distribution $p$ over $[m]$ and
  a vector $\vect{w} \in \mathbb{R}^n$ satisfying
  \begin{equation}\label{eq:cond-unif-w}
    \sum_j{w_j} \geq D^2,
  \end{equation}
  such that for all  $\vect{z} \in \mathbb{R}^n$ 
  \begin{equation}\label{eq:fact-unif}
    \E_{i \sim p}(\sum_{j=1}^n{A_{ij}z_j})^2 \geq \sum_{j=1}^n{w_jz_j^2}.
  \end{equation}
\end{theorem}\qed

We note a geometric interpretation of
Theorem~\ref{thm:vecdisc-unif}. Define the ellipsoids  $E(p, \vect{A}) = \{\vect{z}:
\E_{i \sim p}(\sum_{j=1}^n{A_{ij}z_j})^2 \leq 1\}$ and $F(\vect{w}) =
\{\vect{z}: \sum_{j = 1}^n{w_jz_j^2} \leq 1\}$. Then
Theorem~\ref{thm:vecdisc-unif} states that $\vecdisc(\vect{A}) \geq D$
if and only if there exists a distribution $p$ and $\vect{w}$
satisfying (\ref{eq:cond-unif-w}) such that $E(p, \vect{A}) \subseteq F(\vect{w})$.

\subsection{Linear Algebra}

The following two lemmas are essential to our proof. We suspect they
are standard, but include detailed proofs for completeness. The first
lemma states, geometrically, that any $k$-dimensional section of an
ellipsoid $E$ has volume upper bounded by the volume of the section
with the subspace spanned by the $k$ longest axes of $E$. This fact
follows directly from the Cauchy Interlace Theorem. 

\begin{lemma}[Cauchy Interlace Theorem, see e.g.~Chapter 7 of~\cite{meyer}]
  Let $\vect{X} \in \mathbb{R}^{n \times n}$ be a symmetric real
  matrix with eigenvalues $\lambda_1 \geq \ldots \geq \lambda_n$. Let
  also $\vect{U} \in \mathbb{R}^{n \times k}$ be a matrix with
  mutually orthogonal unit columns. Let finally the eigenvalues of
  $\vect{U^TXU}$ be $\mu_1 \geq \ldots \geq \mu_k$. Then, for all $1
  \leq i \leq k$, $\lambda_{n - k + i}\leq \mu_i \leq \lambda_i$. 
\end{lemma}

\begin{cor}
  \label{cor:det-eig}
  Let $\vect{X} \in \mathbb{R}^{n \times n}: \vect{X} \psd 0$ be a
  symmetric real matrix with eigenvalues $\sigma_1 \geq \ldots \geq \sigma_n \geq
  0$. Let also $\vect{U}\in \mathbb{R}^{n \times k}$ be a matrix with
  mutually orthogonal unit columns. Then $\det(\vect{U^T} \vect{X}
  \vect{U}) \leq \sigma_1 \ldots \sigma_k$. 
\end{cor}

\begin{lemma}
  \label{lm:ineq-det}
  Let $\vect{X} \in \mathbb{R}^{n \times n}: \vect{X} \psd 0$ and 
  $\vect{Y} \in \mathbb{R}^{n \times n}: \vect{Y} \psd 0$ be symmetric
  matrices. Suppose that
  \begin{equation}
    \forall \vect{u} \in \mathbb{R}^n: \vect{u^TXu} \geq \vect{u^TYu}.
  \end{equation}
  Then, $\det(\vect{X}) \geq \det(\vect{Y})$. 
\end{lemma}
\begin{proof}
  For a symmetric real matrix $\vect{M} \psd 0$, define the ellipsoid
  $E(\vect{M}) = \{\vect{u}: \vect{u^TMu} \leq 1\}$. $E(\vect{M})$ is
  unbounded if and only if $\vect{M}$ is singular. Otherwise,
  \begin{equation}\label{eq:vol-ellipsoid}
    \vol(E(\vect{M})) = \frac{\vol(B^n)}{\sqrt{\det(\vect{M})}},
  \end{equation}
  where $B^n$ is the $n$-dimensional unit ball.

  By assumption, $E(\vect{X}) \subseteq E(\vect{Y})$. If
  $\det(\vect{Y}) = 0$, the lemma is trivially true. If
  $\det(\vect{X}) = 0$, then $E(\vect{X})$ is unbounded and therefore
  $E(\vect{Y})$ is unbounded, which implies $\det(\vect{Y}) = 0$. If,
  on the other hand, $E(\vect{X})$ and $E(\vect{Y})$ are bounded, we
  have that $\vol(E(\vect{X})) \leq \vol(E(\vect{Y}))$, and, by
  (\ref{eq:vol-ellipsoid}), $\det(\vect{X}) \geq \det(\vect{Y})$, as
  desired.
\end{proof}

\section{Proof of  Main Theorem}

We begin with an inequality which can be seen as a converse to the
geometric mean--arithmetic mean inequality. The inequality follows
from the Schur convexity of symmetric convex functions; we present a
self-contained elementary proof using a powering trick.
\begin{lemma}
  \label{lm:prod-sum}
  Let $x_1 \geq \ldots \geq x_n > 0$ and $y_1 \geq \ldots \geq y_n >
  0$ such that
  \begin{equation}\label{eq:prod}
    \forall k \leq n: x_1 \ldots x_k \geq y_1 \ldots y_k
  \end{equation}
  Then,
  \begin{equation}\label{eq:sum}
    \forall k \leq n: x_1 + \ldots + x_k \geq y_1 + \ldots + y_k.
  \end{equation}
\end{lemma}
\begin{proof}
  We will show that for all positive integers $L$, $(x_1 + \ldots +
  x_n)^L \geq \frac{1}{n!} (y_1 + \ldots + y_n)^L$. Taking $L$-th
  roots, we get that $x_1 + \ldots +  x_n \geq \frac{1}{(n!)^{1/L}}
  (y_1 + \ldots + y_n)$. Letting $L \rightarrow \infty$ and taking
  limits yields the desired result.

  By the multinomial theorem,
  \begin{equation}
    (x_1 + \ldots +  x_n)^L = \sum_{i_1 + \ldots + i_n = L}{
      \frac{L!}{i_1! \ldots i_n!} x_1^{i_1}\ldots x_n^{i_n}}. 
  \end{equation}
  The inequalities (\ref{eq:prod}) imply that whenever $i_1 \geq
  \ldots \geq i_n$, $x_1^{i_1}\ldots x_n^{i_n} \geq y_1^{i_1}\ldots
  y_n^{i_n}$. Therefore,
  \begin{equation}\label{eq:sumpower-lb}
    (x_1 + \ldots +  x_n)^L \geq \sum_{\substack{i_1 \geq\ldots \geq
        i_n\\i_1 + \ldots + i_n = L}}{
      \frac{L!}{i_1! \ldots i_n!} y_1^{i_1}\ldots y_n^{i_n}}. 
  \end{equation}

  Given a sequence $i_1, \ldots, i_n$, let $\sigma$ be a permutation
  on $n$ elements such that $i_{\sigma(1)} \geq \ldots \geq
  i_{\sigma(n)}$. Since $y_1 \geq \ldots \geq y_n$, we have that
  $y_1^{i_{\sigma(1)}} \ldots y_n^{i_{\sigma(n)}} \geq y_1^{i_1}
  \ldots y_n^{i_n}$. Furthermore, there
  are at most $n!$ distinct permutations of $i_1, \ldots, i_n$ (the
  bound is achieved exactly when all $i_1, \ldots, i_n$ are
  distinct). These observations and the multinomial theorem imply that
  \begin{equation}\label{eq:sumpower-ub}
    (y_1 + \ldots + y_n)^L \leq \sum_{\substack{i_1 \geq\ldots \geq
        i_n\\i_1 + \ldots + i_n = L}}{ \frac{n!L!}{i_1! \ldots i_n!}
      y_1^{i_1}\ldots y_n^{i_n}}. 
  \end{equation}
  Inequalities (\ref{eq:sumpower-lb}) and (\ref{eq:sumpower-ub})
  together imply $(x_1 + \ldots + x_n)^L \geq \frac{1}{n!} (y_1 +
  \ldots + y_n)^L$ as desired.
\end{proof}

We are now ready to prove our main result. 
\begin{theorem}[Theorem~\ref{thm:main} restated]
  For any matrix $\vect{A} \in \mathbb{R}^{m \times n}$ such that
  $\forall i \in [n]: \|\vect{A_{*i}}\| \leq 1$, $\vecdisc(\vect{A})
  \leq 1$. 
\end{theorem}
\begin{proof}
  We will use Theorem~\ref{thm:vecdisc-unif} with $D =
  \sqrt{1+\epsilon}$ for an arbitrary $\epsilon > 0$. For any
  $\vect{w}\in \mathbb{R}^n$ satisfying $\sum_{i = 1}^n{w_i} \geq
  1+\epsilon$ we will show there exists a $\vect{z} \in \mathbb{R}^n$
  satisfying
  \begin{equation}\label{eq:fact-komlos}
    E_{j \sim p}(\sum_{i = 1}^n{A_{ij}z_i})^2 < \sum_i{w_iz_i^2}.    
  \end{equation}
  Therefore, by Theorem~\ref{thm:vecdisc-unif}, $\vecdisc(A)^2 < 1+
  \epsilon$ for all $\epsilon > 0$, which proves our main theorem. 

  For any $i: w_i \leq 0$, we can set $z_i = 0$. Then $\sum_{i: w_i >
    0}{w_i} \geq \sum_i{w_i} \geq 1+\epsilon$. Consider also the
  submatrix $\vect{A'}$ consisting of those columns $\vect{A_{*i}}$ of
  $\vect{A}$ for which $w_i \geq 0$. The matrix $\vect{A'}$ satisfies
  the assumption that all its columns have norm bounded by
  1. Therefore, it is sufficient to show that for any matrix
  $\vect{A}$ with columns bounded by 1 in the euclidean norm, any
  $\vect{w}$ such that $\forall i: w_i > 0$ and $\sum_{i = 1}^n{w_i}
  \geq 1+\epsilon$, and any distribution $p$ on $[m]$, there exists a
  $\vect{z}$ satisfying the bound (\ref{eq:fact-komlos}).
  
  We denote by $\vect{W}$ the diagonal matrix with $\vect{w}$ on the
  diagonal, and similarly for any distribution $\vect{p} \in
  \mathbb{R}_+^m: \sum_{j =1}^m{p_j} = 1$ we denote by $\vect{P}$ the
  diagonal matrix with $\vect{p}$ on the diagonal. In this matrix
  notation, we need to show that for any positive definite diagonal
  matrix $\vect{W}$ such that $\trace(\vect{W}) \geq 1 + \epsilon$,
  and any positive semidefinite diagonal matrix $\vect{P}$ such that
  $\trace(\vect{P}) = 1$, there exists a vector $\vect{z} \in
  \mathbb{R}^n$ such that
  $\vect{z^T}\vect{A^T}\vect{P}\vect{A}\vect{z} <
  \vect{z}\vect{W}\vect{z}$.

  Assume for contradiction that 
  \begin{equation}\label{eq:contr-assump}
    \forall \vect{z}: \vect{z^T}\vect{A^T}\vect{P}\vect{A}\vect{z}
    \geq \vect{z}\vect{W}\vect{z}. 
  \end{equation}
  Geometrically, this is equivalent to $E(p, \vect{A}) \subseteq
  F(\vect{w})$, where $F$ and $E$ are defined as before. The outline
  of our proof is as follows. The relation $F(p, \vect{A}) \subseteq
  E(\vect{w})$ implies that, for all $k$, the largest $k$-dimensional section of $F$
  has volume lower bounded by the volume of the largest
  $k$-dimensional section of $E$. Using Corollary~\ref{cor:det-eig} and the
  Hadamard bound we can show that this implies that, for all $k$, the
  product of the $k$ largest $p_i$ is lower bounded by the product of
  the $k$ largest $w_j$. Then, Lemma~\ref{lm:prod-sum} implies that
  the sum of all $p_i$ is lower bounded by the sum of all $w_j$, which
  is a contradiction. We proceed to prove the above claims formally. 

  Let, without loss of generality, $w_1 \geq \ldots \geq w_n > 0$ and
  similarly $p_1 \geq \ldots \geq p_m \geq 0$. Denote by
  $\vect{A_{[k]}}$ the matrix $(\vect{A_{*1}}, \ldots, \vect{A_{*k}})$
  and by $\vect{W_{k}}$ the diagonal matrix with $w_1, \ldots, w_k$ on
  the diagonal. We first show that
  \begin{equation}
    \label{eq:eig-lb}
    \forall k \leq n: \det(\vect{A_{[k]}^T}\vect{P}\vect{A_{[k]}}) \leq
    p_1\ldots p_k.
  \end{equation}
  Let $\vect{u_1}, \ldots \vect{u_k}$ be an orthonormal basis for the
  range of $\vect{A_{[k]}}$ and let $\vect{U_k}$ be the matrix
  $(\vect{u_1}, \ldots \vect{u_k})$. Then $\vect{A_{[k]}} =
  \vect{U_k}\vect{U_k^T}\vect{A_{[k]}}$. Each column of the square
  matrix $\vect{U_k^T}\vect{A_{[k]}}$ has norm at most $1$, and, by
  Hadamard's inequality, 
  \begin{equation}
    \det(\vect{A^T_{[k]}U_k}) =  \det(\vect{U_k^T}\vect{A_{[k]}}) \leq 1. 
  \end{equation}
  Therefore, 
  \begin{equation}
    \forall k \leq n: \det(\vect{A_{[k]}^T}\vect{P}\vect{A_{[k]}}) \leq
    \det(\vect{U_k^T}\vect{P}\vect{U_k}).
  \end{equation}
  By Corollary~\ref{cor:det-eig}, we have that
  $\det(\vect{U_k^T}\vect{P}\vect{U_k}) \leq p_1 \ldots p_k$, which
  proves (\ref{eq:eig-lb}). 

  By (\ref{eq:contr-assump}) we know that for all $k$ and for all
  $\vect{u} \in \mathbb{R}^k$, $\vect{u^TA_{[k]}^TPA_{[k]}u} \geq
  \vect{u^TW_ku}$, since we can freely choose $\vect{z}$ such that $z_i
  = 0$ for all $i>k$. Then, by Lemma~\ref{lm:ineq-det}, we have that
  \begin{equation}
    \label{eq:eig-ub}
    \forall k\leq n: \det(\vect{A_{[k]}^T}\vect{P}\vect{A_{[k]}}) \geq
    \det(\vect{W_k}) = w_1\ldots w_k
  \end{equation}
  Combining (\ref{eq:eig-lb}) and (\ref{eq:eig-ub}), we have that 
  \begin{equation}
    \label{eq:p_k-w_k}
    \forall k \leq n: p_1 \ldots p_k \geq w_1 \ldots w_k
  \end{equation}
  By Lemma~\ref{lm:prod-sum}, (\ref{eq:p_k-w_k}) implies that $1 = \sum_{j = 1}^m{p_j} \geq
  \sum_{j = 1}^n{p_j} \geq \sum_{i = 1}^n{w_i} \geq 1+ \epsilon$, a
  contradiction. 
  
\end{proof}

\section{Conclusion}

We have shown that the vector discrepancy of a matrix $\vect{A}$ all
of whose columns are contained in the unit ball is bounded by $1$
from above. This result establishes a natural vector discrepancy
variant of the notorious Koml\'{o}s and Beck-Fiala conjectures. On one
hand our result can be seen as evidence in support of the conjectures:
they cannot be disproved by lower bounding vector discrepancy. On the
other hand, our work opens the possibility of giving an efficient
proof of Banaszczyk's bound of $O(\sqrt{\log m})$ on $\disc(\vect{A})$
by improving the pseudoapproximation algorithm of
Bansal~\cite{nikhil}. We hope that our result would prove
useful in an attack on the Koml\'{o}s conjecture itself.

\section*{Acknowledgements}

I would like to thank Nikhil Bansal, Kunal Talwar, Daniel Dadush, and
S.~Muthukrishnan for useful discussions of the result and the
writeup. 

\bibliographystyle{plain}	
\bibliography{disc2}

\begin{thebibliography}{10}

\bibitem{bana}
W.~Banaszczyk.
\newblock Balancing vectors and gaussian measures of n-dimensional convex
  bodies.
\newblock {\em Random Structures \& Algorithms}, 12(4):351--360, 1998.

\bibitem{nikhil}
N.~Bansal.
\newblock Constructive algorithms for discrepancy minimization.
\newblock In {\em Proceedings of the 2010 IEEE 51st Annual Symposium on
  Foundations of Computer Science}, FOCS '10, pages 3--10, Washington, DC, USA,
  2010. IEEE Computer Society.

\bibitem{baranygrinberg}
I.~B\'{a}r\'{a}ny and VS~Grinberg.
\newblock On some combinatorial questions in finite-dimensional spaces.
\newblock {\em Linear Algebra and its Applications}, 41:1--9, 1981.

\bibitem{beckfiala}
J.~Beck and T.~Fiala.
\newblock Integer-making theorems.
\newblock {\em Discrete Applied Mathematics}, 3(1):1--8, 1981.

\bibitem{beck1996discrepancy}
J.~Beck and V.T. S{\'o}s.
\newblock Discrepancy theory.
\newblock In {\em Handbook of combinatorics (vol. 2)}, page 1446. MIT Press,
  1996.

\bibitem{Chazelle}
B.~Chazelle.
\newblock {\em The Discrepancy Method}.
\newblock Cambridge University Press, 1991.

\bibitem{lovasz2000integer}
L.~Lov{\'a}sz.
\newblock Integer sequences and semidefinite programming.
\newblock {\em Publ. Math. Debrecen}, 56:475--479, 2000.

\bibitem{determinant-lb}
L.~Lov{\'a}sz, J.~Spencer, and K.~Vesztergombi.
\newblock Discrepancy of set-systems and matrices.
\newblock {\em European Journal of Combinatorics}, 7(2):151--160, 1986.

\bibitem{lovettmeka}
S.~Lovett and R.~Meka.
\newblock Constructive discrepancy minimization by walking on the edges.
\newblock {\em Foundations of Computer Science, IEEE Annual Symposium on},
  0:61--67, 2012.

\bibitem{matousek2010geometric}
J.~Matousek.
\newblock {\em Geometric Discrepancy: An Illustrated Guide}.
\newblock Springer Verlag, 2010.

\bibitem{matousek-det}
J.~Matou{\v{s}}ek.
\newblock The determinant bound for discrepancy is almost tight.
\newblock {\em Manuscript, Arxiv}, 1101:0767, 2011.

\bibitem{meyer}
Carl~D. Meyer, editor.
\newblock {\em Matrix Analysis and Applied Linear Algebra}.
\newblock Society for Industrial and Applied Mathematics, Philadelphia, PA,
  USA, 2000.

\bibitem{tenlectures}
J.~Spencer.
\newblock {\em Ten Lectures on the Probabilistic Method}.
\newblock SIAM, 1994.

\end{thebibliography}

\end{document}